\documentclass[11pt,a4paper]{article}
\usepackage{a4,amssymb,amsmath,amsthm,url,xcolor,enumerate,float}
\usepackage{bezier,amsfonts,amssymb,graphicx,amsthm,url,lineno,courier}
\usepackage[english]{babel}
\title{Selective Hypergraph Colourings}
\date{}
\begin{document}
\newtheorem{theorem}{Theorem}[section]
\newtheorem{definition}{Definition}[section]
\newtheorem{proposition}[theorem]{Proposition}
\newtheorem{corollary}[theorem]{Corollary}
\newtheorem{lemma}[theorem]{Lemma}
\newcommand*\cartprod{\mbox{ } \Box \mbox{ }}
\newtheoremstyle{break}
  {}
  {}
  {\itshape}
  {}
  {\bfseries}
  {.}
  {\newline}
  {}

\theoremstyle{break}

\newtheorem{propskip}[theorem]{Proposition}
\DeclareGraphicsExtensions{.pdf,.png,.jpg}
\author{Yair Caro \\ Department of Mathematics\\ University of Haifa-Oranim \\ Israel \and Josef  Lauri\\ Department of Mathematics \\ University of Malta
\\ Malta \and Christina Zarb \\Department of Mathematics \\University of Malta \\Malta }

\title{Two short proofs of the Perfect Forest Theorem}
\date{}
\author{Yair Caro$^a$, Josef Lauri$^b$, Christina Zarb$^b$ \footnote{Corresponding Author: Christina Zarb \newline Tel:  00356 25907224 \newline Fax  :00356 21333908 \newline Yair Caro:  yacaro@kvgeva.org.il \newline Josef Lauri:   josef.lauri@um.edu.mt \newline Christina Zarb:   christina.zarb@um.edu.mt  }\\
\\
$^a$  University of Haifa-Oranim, Tivon 36006, Israel.\\
\\
$^b$ University of Malta, Msida MSD 2080, Malta.
 }

\maketitle

\begin{abstract}
A perfect forest is a spanning forest of a connected graph $G$, all of whose components are induced subgraphs of $G$ and such that all vertices have odd degree in the forest. A perfect forest generalised a perfect matching since, in a matching, all components are trees on one edge. Scott first proved the Perfect Forest Theorem, namely, that every connected graph of even order has a perfect forest. Gutin then gave another proof using linear algebra.

We give here two very short proofs of the Perfect Forest Theorem which use only elementary notions from graph theory. Both our proofs yield polynomial-time algorithms for finding a perfect forest in a connected graph of even order.

\end{abstract}

\section{Introduction}

We consider a graph $G$ on $n$ vertices.   A spanning subgraph $F$ of a graph $G$ is called a \emph{perfect forest} if
\begin{itemize}
\item{$F$ is a forest}
\item{the degree of every vertex in $F$ is odd}
\item{each subtree of $F$ is an induced subgraph of $G$.}
\end{itemize}

An important graph structure is a \emph{matching}, which  is a set of pairwise non-adjacent edges, that is  without common vertices.  A \emph{maximum matching} is a matching that contains the largest possible number of edges, and such a matching is said to be \emph{perfect} if it includes all the vertices in the graph \cite{lovasz2009matching}.  It is well-known that a perfect/maximum matching in a graph can be found in polynomial time \cite{Galil}.

In the case of a connected graph $G$ on $n$ vertices, $n$ is even, without a perfect matching, one would like to find other structures that are in a sense reminiscent of perfect matchings.   Clearly in a perfect matching each edge is induced and both degrees are odd.  This leads to the above definition of perfect forests.  Establishing the existence of a perfect forest, and finding one in  polynomial time are comparable to such questions with respect to perfect matchings.
 
Taking another point of view, the question of the maximum size of an induced subgraph of $G$ with all degrees odd has also been treated in several papers, and is also related to perfect forests \cite{caro1994induced,radclife1995every,scott1992large}. 

Scott \cite{scott2001induced} proved that every connected graph $G$ with an even number of vertices contains a perfect forest.  We call this the Perfect Forest Theorem. Scott's relatively long proof is graph theoretical and constructive.  Gutin \cite{gutin2015note}  provides a shorter proof using simple linear algebra arguments, which in turn can give a polynomial-time algorithm to find a perfect forest.

Gutin and Yeo have considered perfect forests in directed graphs in \cite{g2015note}.  They consider four generalizations to directed graphs of the concept of a perfect forest and  extend Scott's theorem to digraphs in a non-trivial way.

An interesting algorithmic application of perfect forests can be found in \cite{wigderson1996new}, where Sharan and Wigderson use perfect forests, which they call pseudo-matchings, in an algorithm which yields a perfect matching in a bipartite cubic graph.  This again justifies finding short proofs for the Perfect Forest Theorem, which yield a polynomial time algorithm.

Motivated by this interest, we provide two simple graph theoretic proofs of the Perfect Forest Theorem, one by induction on the number of vertices, and the other by induction on the number of edges.  Both proofs yield a polynomial-time algorithm for finding a perfect forest in a graph with an even number of vertices.

In our proofs, we define the \emph{union} of two graphs $G_1$ and $G_2$, denoted by $G_1 \cup G_2$, as the graph with vertex set $V(G_1) \cup V(G_2)$ and edge set $E(G_1) \cup E(G_2)$.

\section{Two short proofs}

We start with a simple lemma which will be used in the first proof.

\begin{lemma} \label{evendegree}
Let $G$ be a connected graph on $n$ vertices where $n$ is even, and $G$ is not a tree.  Then $G$ has a spanning tree  with at least two vertices of $even$ degree.
\end{lemma}

\begin{proof}
Let $T$ be a spanning tree of $G$.  Then clearly, the number of vertices of even degree must be even since the total number of vertices is even.

Let us assume that all vertices of $T$ have odd degree.  Since $G$ is not a tree, there must be at least one edge $e \in E(G\backslash T)$.  Let this edge join the vertices $u$ and $v$.  Consider the tree $T^*$ obtained from $T$ by adding the edge $e$ and deleting the edge connecting $u$ to $w$ where $w$ is the unique neighbour of $u$ on the path from $u$ to $v$ in $T$. Then the degree of $v$ and $w$ are now even, and the spanning tree $T^*$ is as required.
\end{proof}

\begin{theorem}
Let $G$ be a connected graph on $n$ vertices where $n$ is even.  Then $G$ has a perfect forest.
\end{theorem}

\begin{proof} 

{\bf First proof.}

\medskip
\noindent
We prove this by induction on $n$.  For $n=2$, the statment is clearly true.  So suppose $n \geq 4$ is even.
\begin{enumerate}
\item{If $G$ is a tree with all degrees odd, then $G$ is the required perfect forest.}
\item{If $G$ is a tree with at least one vertex $w$ of even degree, let the branches of $w$ be $B_1, \ldots, B_t$, where $t$ is even.  Then at least one of these branches, say $B_1$, contains an even number of vertices, otherwise the number of vertices in $G$ is odd. Let $G_1$ be the induced subgraph of $G$ on $B_1$, and $G_2$ the subgraphs induced by $w$ and $B_2,\ldots,B_t$.  Clearly $G_1$ and $G_2$ are both connected and have even order, and hence by induction, there exists perfect forests $F_1$ and $F_2$ in $G_1$ and $G_2$ respectively.  Hence, $F_1 \cup F_2$ is a perfect forest of $G$.}
\item{If $G$ is not a tree, then by Lemma \ref{evendegree},  $G$ has  a spanning tree $T$ with at least two vertices of even degree, and hence by the same argument presented in (2) above, $T$ enables us to split $G$ into two vertex-disjoint connected components each of even order, $G_1$ and $G_2$ , and applying induction on both components gives a perfect forest.}
\end{enumerate}
\emph{Remark:} This proof tells us that either $G$ is a tree with all vertices having odd degree or $V(G)$ can be split as $V=A\cup B$, $A\cap B=\emptyset$ and both $A$ and $B$ induce  connected components each of even order (hence induction applies) .

\medskip
Now we give the second proof.

\medskip
\noindent
{\bf Second Proof.}

\medskip
\noindent
We now use induction on the number of edges $m$.  The case $m=1$ is trivial.  So assume $m \geq 2$ and $n$ is even.  Consider $G-e$, for some $e \in E(G)$.
\begin{enumerate}
\item{Consider the case where  $G-e$ is disconnected.  If the two components $G_1$ and $G_2$ each have an even number of vertices, then by the induction hypothesis, they each have a perfect forest $F_1$ and $F_2$ respectively, and $F_1 \cup F_2$ is a perfect forest in $G-e$ and hence in $G$.

If $G_1$ and $G_2$ each have an odd number of vertices, then consider $G_1'=G_1+e$ and $G_2'=G_2+e$.  Each has an even number of vertices because we are adding a new vertex to each, and hence by the induction hypothesis,  they each have a perfect forest $F_1'$ and $F_2'$ respectively.  But then  $F_1' \cup F_2'$ is a perfect forest for  $G$.}
\item{If $G-e$ is still connected, by the induction hypothesis there exists a perfect forest $F$ in $G-e$.  If the edge $e$ joins different components of $F$, then $F$ is still a perfect forest in $G$.  If, on the other hand, $e$ joins two vertices, say $u$ and $v$,  in the same component $H$ of $F$, let $K$ be the union of the other components of $F$ (we are not excluding the possibility that $K$ is empty, that is, $F=H$, that is, $F$ is a tree). 

Now consider the subgraph induced in $G$ by the vertices of $H$.  The key fact here is that this induced subgraph of $G$ contains no edges other than those in $E(H)$ except $e$, since $H$ is a connected induced subgraph of $G-e$. 

Therefore $H+e$ is an induced subgraph of $G$. This induced subgraph $H+e$ now has exactly one cycle $C$ which contains the edge $e$.  Consider $H-E(C)$, that is, the subgraph of $H$ obtained by removing all the edges of this cycle.  Now $H-E(C)$ contains no cycles so it is forest all of whose components are induced sub-trees of $G$ because every edge in $E(C)$ joins two vertices in different components of $H-E(C)$. Now, the degrees of all vertices in $H-E(C)$ except $u$ and $v$ are reduced by 2 when the edges in $E(C)$ are removed. Therefore they are still odd in $H-E(C)$ since they were odd in $H$. The degrees of $u$ and $v$ in $H-E(C)$ are, however, one less than in $H$ since they are reduced by 1 when removing the edges of $E(C)$.  Of course the degrees of the other vertices of $F$, those in $K$, are unchanged. 

So we put back the edge $e$ into $H-E(C)$, giving the graph $H'$, in which all vertices, including $u$ and $v$, have odd order. Also, $H'$ is induced in $G$ by the key fact pointed out above.  Therefore $F'$, the forest containing the union of $H'$ and $K$ rather than $H$ and $K$, is a perfect forest in $G$.}
\end{enumerate}
\end{proof}

\emph{Remark:} Clearly, both  proofs are constructive and can be implemented in polynomial time.

\section{Conclusion}

Perfect forests have attracted the attention of various graph theorists as one natural extension of the idea of perfect matchings. In the literature two proofs of the Perfect Forest Theorem were known, one using techniques from linear algebra. We here give two very short proofs which use only elementary notions from graph theory, both of which yield polynomial-time algorithms for finding a perfect forest in a graph of even order.

\bibliographystyle{plain}
\bibliography{bibtrees}

\end{document}